\documentclass[12pt,a4paper,oneside]{amsart}

\usepackage{amsfonts, amsmath, amssymb, amsthm, amscd, hyperref}
\usepackage{anysize}
\usepackage{graphicx}
\usepackage{pdfsync}
\usepackage[all]{xy}
\usepackage{enumitem}

\newtheorem{theorem}{Theorem}[section]
\newtheorem*{theorem*}{Theorem}
\newtheorem{lemma}[theorem]{Lemma}

\newtheorem*{protoquestion*}{Prototype question}

\theoremstyle{definition}
\newtheorem{definition}[theorem]{Definition}

\theoremstyle{remark}
\newtheorem{remark}[theorem]{Remark}
\newtheorem*{remarks*}{Remarks}

\newtheorem{example}[theorem]{Example}

\numberwithin{equation}{section}

\DeclareMathOperator{\diam}{diam}

\DeclareMathOperator{\rk}{rk}
\DeclareMathOperator{\dist}{dist}

\DeclareMathOperator{\UW}{UW}

\DeclareMathOperator{\tc}{tc}
\DeclareMathOperator{\W}{W}
\DeclareMathOperator{\Z}{\mathcal{Z}}
\DeclareMathOperator{\Sig}{\mathfrak{S}}

\newcommand{\eps}{\varepsilon}

\begin{document}

\title{Waist of maps measured via Urysohn width}

\author{Alexey~Balitskiy{$^\clubsuit$}}

\email{{$^\clubsuit$}balitski@mit.edu}

\author{Aleksandr~Berdnikov{$^\spadesuit$}}

\email{{$^\spadesuit$}aberdnik@mit.edu}

\address{{$^\clubsuit$}{$^\spadesuit$} Dept. of Mathematics, Massachusetts Institute of Technology, 182 Memorial Dr., Cambridge, MA 02142, USA}
\address{{$^\clubsuit$} Institute for Information Transmission Problems RAS, Bolshoy Karetny per. 19, Moscow, Russia 127994}



\begin{abstract}
We discuss various questions of the following kind: for a continuous map $X \to Y$ from a compact metric space to a simplicial complex, can one guarantee the existence of a fiber large in the sense of Urysohn width? The $d$-width measures how well a space can be approximated by a $d$-dimensional complex.
The results of this paper include the following.
\begin{enumerate}
  \item Any piecewise linear map $f: [0,1]^{m+2} \to Y^m$ from the unit euclidean $(m+2)$-cube to an $m$-polyhedron must have a fiber of $1$-width at least $\frac{1}{2\beta m +m^2 + m + 1}$, where $\beta = \sup\limits_{y\in Y} \rk H_1(f^{-1}(y))$ measures the topological complexity of the map.
  \item There exists a piecewise smooth map $X^{3m+1} \to \mathbb{R}^m$, with $X$ a riemannian $(3m+1)$-manifold of large $3m$-width, and with all fibers being topological $(2m+1)$-balls of arbitrarily small $(m+1)$-width.
\end{enumerate}
\end{abstract}

\maketitle

\setcounter{section}{-1}
\section{Introduction}
\label{sec:intro}

The notion of the \emph{Urysohn width} of a compact metric space was suggested by Pavel Urysohn in 1920s (and published posthumously by Pavel~Alexandrov~\cite{alexandroff1926notes}). The $d$-width measures how well a space can be approximated by a $d$-dimensional simplicial complex. A compact metric space $X$ is said to have \emph{$d$-width} at most $w$, if there is a continuous map $X \to Z^d$ to a $d$-dimensional simplicial complex with all fibers having diameter at most $w$. The original definition of Urysohn was given in terms of closed coverings, and we give an overview of different equivalent ways of defining width in Section~\ref{sec:urysohn}.

The Urysohn width of a riemannian manifold is related to other metric invariants. For example, the codimension $1$ width does not exceed the $n\textsuperscript{th}$ root of the volume (up to a dimensional factor; see~\cite{guth2017volumes}), and bounds from above the filling radius of a manifold (see~\cite[Appendix~1]{gromov1983filling}) and its hypersphericity (see~\cite[Proposition~F$_1$]{gromov1988width}~or~\cite[Section~5]{guth2005lipshitz}). Among the applications of the Urysohn width we mention a recent transparent proof~\cite{nabutovsky2019linear} of Gromov's systolic inequality, building on the ideas from~\cite{papasoglu2020uryson, guth2010systolic}.

The question raised in this paper is inspired by another famous Gromov's inequality, namely the waist of the sphere theorem~\cite{gromov2003isoperimetry}. It says that any generic smooth map $f: S^n \to \mathbb{R}^m$, $m<n$, has a fiber of $(n-m)$-volume at least the one of the $(n-m)$-dimensional ``equatorial'' subsphere. The target space can be replaced by any $m$-manifold~\cite{karasev2013waist}, while it is not clear if one can replace it by an $m$-polyhedron $Y^m$. The only result in this direction we are aware of is~\cite[Theorem~7.3]{akopyan2012borsuk}, saying that any generic smooth map $S^n \to Y^{n-1}$ has a fiber of length $\ge \pi$. A non-sharp version of the waist theorem, however, can be proved for any $m$-dimensional target space by induction using the Federer--Fleming isoperimetric estimate. This type of argument apparently goes back to Almgren, and it was used by Gromov in~\cite{gromov1983filling} (see the exposition in~\cite[Section~7]{guth2008waist}, which applies to any target space, or in~\cite[Section~7]{akopyan2019lower}). A discrete version of this non-sharp estimate is proven in~\cite{matdinov2013size} along the same lines. For Riemannian metrics other than round, the case $n=2$ is understood~\cite{liokumovich2014slicing, balacheff2015measurements}, and the case $n=3$ is investigated under additional curvature assumptions~\cite{liokumovich2020waist}.

The Urysohn width itself is a waist-type invariant, in which the size of a fiber is measured via its diameter, instead of the volume. In this paper, we investigate (non-sharp) waist theorems, where the size of a fiber is measured via the Urysohn width.

\begin{protoquestion*}
\label{ques:waist}
Fix integers $n,m,d$, such that $n > m + d$. Let $f: X^n \to Y^m$ be a continuous map from a compact riemannian $n$-manifold to an $m$-dimensional simplicial complex. Let $w$ be the supremal Urysohn $d$-width of fibers $f^{-1}(y)$, $y \in Y$, viewed as compact metric spaces with the extrinsic metric of $X$. Can one bound $w$ from below in terms of the $(n-1)$-width of $X$? If not, can one bound $w$ if the ``topological complexity'' of the fibers is restricted?
\end{protoquestion*}

It is natural to expect that the answer should be affirmative in some sense when $n > m + d$ (if we hope that the corresponding property of the dimension~\cite{hurewicz1941dimension, balka2017dimensions} is robust). When $d=1$, and the first Betty number of the fibers is bounded, this is indeed the case, as we will show in Section~\ref{sec:boundedfibers}. However, in general this is far from true. In Section~\ref{sec:sasha} it will be shown that even for $n = (m+1)(d-m)+2m$ and topologically trivial fibers the answer is negative. In a sense, this shows the failure of the notion of the $d$-width to measure the ``defect of $d$-dimensionality''.

Let us describe the answers for the first four non-trivial cases of Prototype question. These four claims are the simplest special cases of the theorems explained in this paper.

\begin{enumerate}[label=(\Alph*)]
  \item\label{cl:A} \textit{There is a map $f: [0,1]^3 \to [0,1]$ with all fibers having arbitrarily small $1$-width.}

  We describe this example (\cite[Example~H$_1''$]{gromov1988width}) briefly. Consider an $\eps$-fine cubical grid in $\mathbb{R}^3$, and let $Z_0$ be its $1$-skeleton. Let $Z_1$ be the $1$-skeleton of the dual grid. Define $f$ by setting $f(x) = \frac{\dist(x,Z_0)}{\dist(x,Z_0) + \dist(x,Z_1)}$. It can be checked that every fiber $\Sigma_y = f^{-1}(y)$, $y \in [0,1)$, retracts to $Z_0$ with every point moving by distance $\lesssim \eps$; hence it has small $1$-width. Similarly, the fibers over $y \in (0,1]$ are approximated by $Z_1$.

  In Section~\ref{sec:arbitraryfibers}, we explain how this example is generalized to higher dimensions, see Theorem~\ref{thm:waistunconstrainedsharp}. This might be known to experts, but we were not able to locate a reference.

  \item\label{cl:B} Notice that all regular fibers in the previous example have high genus. What happens if we bound their topological complexity?

\textit{Suppose that a piecewise linear map $f: [0,1]^3 \to [0,1]$ is such that all fibers $f^{-1}(y), y \in [0,1]$, are homeomorphic to $[0,1]^2$. Then there is a fiber $f^{-1}(y)$ of Urysohn $1$-width at least $\frac{1}{3}$.}

This is the baby case of one of our main results, Theorem~\ref{thm:waist}. Here is the idea of the proof that will be developed in Section~\ref{sec:boundedfibers}. Suppose that every fiber $X_y = f^{-1}(y)$ has width $\UW_{d}(X_y) < c$. So there are maps $X_y \to Z_y$ to graphs $Z_y$ whose fibers are of diameter less than $c$. A na\"{\i}ve idea might be to assemble them together to get a map $[0,1]^3 \to \bigcup Z_y$. If there was a nice way to interpret $\bigcup Z_y$ as a two-dimensional space, then we would be done as long as $c < \UW_{n-1}(X)$. A careful argument might try to assemble the maps $X_y \to Z_y$ by induction on the skeletal structure of $Y$, subdivided finely. The newly built intermediate maps will have fibers with the size bounded in terms of $c$ and the ``topological complexity'' of the $X_y$.

  \item\label{cl:C} The following is a special case of~\cite[Corollary~H$_1'$]{gromov1988width}, which we discuss in Section~\ref{sec:arbitraryfibers} (see Theorem~\ref{thm:waistunconstrained}).

        \textit{Every continuous map $f: X^4 \to Y^1$ from a compact metric space to a graph has a fiber whose $1$-width is at least the $3$-width of $X$.}
  \item\label{cl:D} Another major result of this paper is Theorem~\ref{thm:bundle}, a family of examples of maps with small and topologically trivial fibers; here is the simplest case.

      \textit{There is a riemannian metric on $[0,1]^4$ that has substantial $3$-width but the fibers of the coordinate projection $f : [0,1]^4 \to [0,1]$ all have small $2$-width.}

      We sketch roughly the idea of the construction. For each $y \in [0,1]$, the standard metric inside the fiber $f^{-1}(y) \simeq [0,1]^3$ is modified as follows. Consider the high-genus surface $\Sigma_y \subset f^{-1}(y)$, as in the example~\ref{cl:A}. In its small tubular neighborhood, blow up the metric in the normal direction; then, squeeze the metric everywhere outside the tubular neighborhood. The result can be mapped to the suspension of $Z_0$ or $Z_1$ with small fibers. However, the entire space $[0,1]^4$ can be shown to have substantial $3$-width.
\end{enumerate}

It is worth noting that the theorems explained in this paper do not answer Prototype question completely. For example, imagine one has a map $X^{k^2} \to Y^{k-1}$ with topologically trivial fibers ($k \ge 3$), and with $\UW_{k^2-1}(X^{k^2}) = 1$. Theorem~\ref{thm:bundle} implies that it might happen that all the fibers have $(2k-1)$-width arbitrarily small. Theorem~\ref{thm:waistunconstrained} implies that there is a fiber $F$ with $\UW_{k-1}(F) \ge 1$. It is not clear though what can be said about the widths $\UW_d$ of the fibers in the range $k-1 < d < 2k-1$.

\subsection*{Acknowledgements} We thank Larry Guth for helpful discussions. This paper arose from a chapter of the thesis defended by the first-named author under the supervision of Larry Guth at MIT.

\section{Urysohn width}
\label{sec:urysohn}

Everywhere in this section, $X$ denotes a compact metric space. The diameter of a set is measured using the distance function in $X$: $\diam A = \sup\limits_{a,a' \in A} \dist_X(a,a')$.

\begin{definition}
\label{def:urysohn}
The \emph{Urysohn $d$-width} of a closed subset $S$ of a compact metric space $X$ can be defined in either of the following ways.
\begin{equation*}
\label{eq:uo}
\UW_d(S) = \inf\limits_{\bigcup U_i \supset S} \sup\limits_{i} \diam(U_i), \tag{UO}
\end{equation*}
where the infimum is taken over all open covers of $S$ of multiplicity at most $d+1$.
\begin{equation*}
\label{eq:uc}
\UW_d(S) = \inf\limits_{\bigcup C_i = S} \sup\limits_{i} \diam(C_i), \tag{UC}
\end{equation*}
where the infimum is taken over all finite closed covers of $S$ of multiplicity at most $d+1$.
\begin{equation*}
\label{eq:um}
\UW_d(S) = \inf\limits_{p: S \to Z} \sup\limits_{z \in Z} \diam(p^{-1}(z)), \tag{UM}
\end{equation*}
where the infimum is taken over all continuous maps $p$ from $S$ to any metrizable topological space $Z$ of covering dimension at most $d$.

The quantity $\W(p) = \sup\limits_{z \in Z} \diam(p^{-1}(z))$ will be called the \emph{width} of the map $p$.
\end{definition}


The class of test spaces $Z$ in~\eqref{eq:um} can be narrowed down to $d$-dimensional simplicial complexes, without changing the width, as it will implicitly follow from the proof below.

\begin{proof}[Proof of the equivalence of different definitions of the Urysohn width]$ $\newline
Denote by $w_\text{c}, w_\text{o}, w_\text{m}$ the width of a set $S \subset X$ measured as in~\eqref{eq:uc},~\eqref{eq:uo},~\eqref{eq:um}, respectively.
\begin{itemize}
\item[(\ref{eq:uo} $\le$ \ref{eq:uc})]
Given a finite closed covering $S = \bigcup C_i$, we can use compactness to argue that
\[
\delta_{i_0, \ldots, i_{d+1}} = \min\limits_{x\in X} \max_{0\le k\le d+1} \dist(x, C_{i_k})
\]
is attained and positive.
Take $\eps > 0$ smaller then each $\delta_{i_0, \ldots, i_{d+1}}$ over all collections of indices $i_0 < \ldots < i_{d+1}$, and consider the open covering $\{U_i\}$, where $U_i = U_\eps(C_i)$ is the $\eps$-neighborhood of $C_i$. It has the same multiplicity as the covering $\{C_i\}$, and $\max \diam U_i \le \max \diam C_i + 2\eps$. Taking $\eps \to 0$, we get $w_\text{o} \le \max \diam C_i$. Therefore, $w_\text{o} \le w_\text{c}$.
\item[(\ref{eq:uc} $\le$ \ref{eq:um})]
Suppose we are given a map $p: S \to Z^d$ to a metrizable space; fix a metric on $Z$. Recall that the width of $p$ is defined as $\W(p) = \sup_{z \in Z} \diam (p^{-1}(z))$. Fix a small number $\eps > 0$. For each point $z \in p(S)$ one can find radius $r(z) > 0$ such that the preimage of $V_{r(z)}(z)$, the $r(z)$-neighborhood of $z$, has diameter smaller than $\W(p) + \eps$. Here we used
\[
\lim\limits_{r \to 0} \diam(p^{-1}(V_{r}(z))) = \diam(p^{-1}(z)).
\]
By definition of dimension (and compactness), there is a finite open covering $\{V_i\}$ of $p(S)$, refining $\{V_{r(z)}(z)\}$, and with multiplicity at most $d+1$. It follows from Lebesgue's number lemma that there is a closed covering $\{D_i\}$ with $D_i \subset V_i$. Then the closed sets $C_i = p^{-1}(D_i)$ have diameter less than $\W(p) + \eps$, and cover $S$ with multiplicity at most $d+1$. Repeating this with arbitrarily small $\eps$, one gets $w_\text{c} \le \W(p)$. Since this is true for all $p$, we conclude $w_\text{c} \le w_\text{m}$.
\item[(\ref{eq:um} $\le$ \ref{eq:uo})]
Given an open covering $S \subset \bigcup U_i$ (which we can assume finite by compactness) with multiplicity $d+1$, consider the mapping to its nerve
\[
\varphi : S \to N^d,
\]
associated to any subordinate partition of unity. The preimage of every point is entirely contained in some $U_i$, hence $\W(\varphi) \le \sup \diam U_i$. Therefore, $w_\text{m} \le w_\text{o}$.
\end{itemize}
\end{proof}

Definition~\ref{def:urysohn} was given for a closed set $S$. We adopt the following convention: the width of a (not necessarily closed) set $S \subset X$ is defined in terms of open coverings,~\eqref{eq:uo}.

\begin{lemma}
\label{lem:semicontinuous}
Let $f : X \to Y$ be a continuous map from a compact metric space $X$ to a metrizable topological space $Y$. The function
\[
y \mapsto \UW_d(f^{-1}(y))
\]
is upper semi-continuous for any $d$. Namely,
\[
\UW_d(f^{-1}(y)) \ge \limsup_{y' \to y} \UW_d(f^{-1}(y')).
\]
\end{lemma}

\begin{proof}
If a fiber $f^{-1}(y)$ is covered by open sets $U_i \subset X$, with diameters $<\UW_d(f^{-1}(y))+\eps$ and multiplicity at most $d+1$, then these open sets in fact cover neighboring fibers $f^{-1}(y')$ as well.
\end{proof}

\section{Waist of maps with arbitrary fibers}
\label{sec:arbitraryfibers}

\begin{theorem}[{\cite[Corollary~H$_1'$]{gromov1988width}}]
\label{thm:waistunconstrained}
Let $X$ be a compact metric space, and let $Y$ be a metrizable 
topological space of covering dimension $m$. Every continuous map $f: X \to Y$ has a fiber $f^{-1}(y)$ of $d$-width $\UW_{d}(f^{-1}(y)) \ge \UW_{n-1}(X)$, where $n = (m+1)(d+1)$.
\end{theorem}

\begin{proof}
The assumptions on $Y^m$ imply that $\UW_d(f^{-1}(y)) = \inf\limits_{\text{open } V \ni y} \UW_d(f^{-1}(V))$. Supposing the contrary to the statement of the theorem, and pulling back a fine open cover of $Y$, we obtain an open cover $\{U_i\}$ of $X$ of multiplicity at most $m+1$, such that $\UW_{d}(U_i) < u := \UW_{n-1}(X)$ for all $i$. It follows from the definition of the $d$-width that every $U_i$ admits an open cover $U_i = \bigcup\limits_j U_{ij}$ of multiplicity at most $d+1$, with $\diam U_{ij} < u$. The cover $\{U_{ij}\}$ of $X$ has multiplicity at most $(m+1)(d+1)$, and it can be assumed finite (by compactness), so we get $\UW_{n-1}(X) < u$, which is absurd.
\end{proof}

The relation between dimensions $n,m,d$ in Theorem~\ref{thm:waistunconstrained} is optimal, as the following result (generalizing example~\ref{cl:A} from the introduction) shows.

\begin{theorem}
\label{thm:waistunconstrainedsharp}
Let $n = (m+1)(d+1) - 1$, and let $\eps > 0$ be any small number. There exists a continuous map $f: B^n \to \triangle^m$ from the unit euclidean $n$-ball to the $m$-simplex, whose fibers all have Urysohn $d$-width less than $\eps$.
\end{theorem}

\begin{remark}
\label{rem:ballwidth}
It is easy to show that $\UW_{n-1}(B^n) > 0$. This can be deduced from the Lebesgue covering theorem~\cite{lebesgue1911non, brouwer1913naturlichen}, or from the Knaster--Kuratowski--Mazurkiewicz theorem~\cite{knaster1929beweis}. In fact, the exact value $\UW_{n-1}(B^n) = \sqrt{\frac{2n+2}{n}}$ is known (see~\cite[pp.~84--85,~268]{tikhomirov1976some}~or~\cite[Remark~6.10]{akopyan2012borsuk}).
\end{remark}

The crucial tool used in the proof of Theorem~\ref{thm:waistunconstrainedsharp} is the \emph{local join representation} of $\mathbb{R}^n$, which will be also used in Section~\ref{sec:sasha}.

\begin{lemma}[{cf.~\cite[Lemma~4.1]{balitskiy2020local}}]
\label{lem:triangulation}
Fix $\eps > 0$. There is a locally finite triangulation of $\mathbb{R}^n$ by simplices of diameter $< \eps$, admitting a nice coloring: the vertices receive colors $0, 1, \ldots, n$ so that each simplex receives all distinct colors.
\end{lemma}

\begin{proof}
In fact, there is such a triangulation with simplices congruent to one another, via the reflection in the facets. Such a triangulation can be obtained from the type $A$ root system and the corresponding affine Coxeter hyperplane arrangement (see~\cite[Chapter~6]{shi1986kazhdan}). (Of course, simpler constructions are also possible.)
\end{proof}

\begin{definition}[{cf.~\cite[Definition~4.2]{balitskiy2020local}}]
\label{def:join}
Let $n = (m+1)(d+1) - 1$, and triangulate $\mathbb{R}^n$ by $\eps$-small simplices, as in Lemma~\ref{lem:triangulation}. Define $Z_i$, $0 \le i \le m$, to be the union of all simplices of the triangulation colored by colors $(d+1)i$ through $(d+1)i+d$. We say that $\mathbb{R}^n$ is the \emph{$\eps$-local join} of $d$-dimensional complexes $Z_0, \ldots, Z_m$.
\end{definition}
The name is justified by the following observation: every (top-dimensional) simplex $\sigma$ of the triangulation can be written as the join $(\sigma \cap Z_0) * \ldots * (\sigma \cap Z_m)$; that is, any point $x \in \sigma$ can be written as
\[
x = \sum\limits_{i=0}^m t_i z_i, \quad \text{where } z_i \in \sigma \cap Z_i, ~~ t_i \ge 0, ~~ \sum_{i=0}^m t_i = 1.
\]
The coefficients $t_i$ are determined uniquely, giving a well-defined \emph{join map}
\[
\tau : \mathbb{R}^n \to \triangle^{m} = \left\{(t_0, \ldots, t_m) ~\middle\vert~ t_i \ge 0, ~ \sum_{i=0}^m t_i = 1 \right\}.
\]
Note that $Z_i = \tau^{-1}(v_i)$, where $v_0, \ldots, v_m$ are the vertices of $\triangle^{m}$. For each vertex $v_i$, denote the opposite facet of $\triangle^{m}$ by $v_i^\vee$. For each complex $Z_i$, its \emph{dual} $(md + m -1)$-dimensional complex is given by $Z_i^\vee = \tau^{-1}(v_i^\vee)$. There are natural retractions
\[
\pi_i :  \mathbb{R}^n \setminus Z_i^\vee \to Z_i,
\]
defined by sending $x = \sum\limits_{i=0}^m t_i z_i \in \sigma$ to $z_i \in \sigma \cap Z_i$; they are well-defined since $t_i \neq 0$ whenever $x \notin Z_i^\vee$. Note that $\pi_i$ moves each point by distance $< \eps$.

\begin{proof}[Proof of Theorem~\ref{thm:waistunconstrainedsharp}]
Represent $\mathbb{R}^n$ as the \emph{$\eps/2$-local join} of $d$-dimensional complexes $Z_0, \ldots, Z_m$; let $\tau : \mathbb{R}^n \to \triangle^m$ be its join map. Take $f$ to be the restriction of $\tau$ on the unit ball $B^n$. Let us check that the $d$-width of any fiber $F = f^{-1}(t_0, \ldots, t_m)$ is small. Fix any $i$ for which $t_i \neq 0$. The (restricted) retraction map $\pi_i\vert_{F} :  F \to Z_i$ has fibers of diameter $<\eps$, so we are done.
\end{proof}

\section{Waist of maps with fibers of bounded complexity}
\label{sec:boundedfibers}

This section generalizes example~\ref{cl:B} from the introduction. The main result, Theorem~\ref{thm:waist}, in particular implies the following waist inequality.

\textit{Any piecewise linear map $f: X^{m+2} \to Y^m$ from a metric $(m+2)$-polyhedron to an $m$-polyhedron must have a fiber of $1$-width at least $\frac{\UW_{m+1}(X)}{2\beta m +m^2 + m + 1}$, where $\beta = \sup\limits_{y\in Y} \rk H_1(f^{-1}(y))$ measures the topological complexity of the map.}

\subsection{PL maps of polyhedra.}


We use the word \emph{polyhedron} to refer to a topological space admitting a structure of a finite simplicial complex (together with rectilinear structure on each simplex), though we do not usually specify this structure. We say a continuous map $X \to Y$ of polyhedra is a \emph{piecewise linear map}, or a \emph{PL map}, if it is simplicial for some fine simplicial structures on $X$ and $Y$.

A polyhedron with a metric space structure (giving the same topology) will be called a \emph{metric polyhedron}. For example, it could be a polyhedron endowed with a smooth riemannian metric on each maximal simplex, so that the metrics on adjacent simplices match in restriction to their common face.

For a map $f: X \to Y$, we sometimes denote the preimage $f^{-1}(A)$ of a subset $A \subset Y$ by $X_A$, if there is no confusion and $f$ is understood from the context. If $X$ and $A \subset Y$ are polyhedra, and $f$ is a PL map, then $X_A$ is naturally a polyhedron. If additionally $X$ is metric, then $X_A$ is metric as well (with the extrinsic metric).

\begin{definition}
\label{def:topcomplexity}
We measure the \emph{topological complexity} using the first Betty number. For a space $X$, we set $\tc(X) = \rk H_1(X; \mathbb{Z}) = \dim H_1(X; \mathbb{Q})$. For a map $f: X \to Y$, we set $\tc(f) = \sup\limits_{y \in Y} \tc(X_y)$.
\end{definition}

\begin{remark}
\label{rem:topcomplexity}
In fact, the estimates~\ref{lem:inter},~\ref{lem:parainter},~\ref{thm:waist},~\ref{lem:intersimplex} of this section hold in a stronger form, with $\tc(\cdot)$ replaced by a smaller quantity. Namely, we define $\tc'(X)$ as the largest number of linearly independent classes in $H^1(X; \mathbb{Q})$ with pairwise zero cup-products. Similarly, for a map $f: X \to Y$, we set $\tc'(f) = \sup\limits_{y \in Y} \tc'(X_y)$. We formulate our results with $\tc(\cdot)$ for simplicity, but in the proofs we indicate the adjustments needed if we use $\tc'(\cdot)$.
\end{remark}

\begin{example}
\label{ex:topcomplexity}
If $X$ is a closed connected oriented surface of genus $g$, then $\tc(X) = 2g$ while $\tc'(X) = g$. If $X$ is a connected oriented surface of genus $g$ with $q > 0$ punctures, then $\tc(X) = \tc'(X) = 2g+q-1$.
\end{example}

\begin{lemma}
\label{lem:plbundle}
Every PL map $f: X \to Y$ of polyhedra satisfies the following regularity assumption. Fix a simplicial structure on $Y$ for which $f$ is simplicial. Fix a simplex $\triangle \subset Y$ (of any dimension), and let $\mathring{\triangle}$ be its relative interior. Then one can pick a PL map $\Psi_{\triangle}: \triangle \times \Sigma_\triangle \to X_{\triangle}$, for some polyhedron $\Sigma_\triangle$, such that
\begin{itemize}
  \item $\Psi_\triangle$ is fibered over $\triangle$:
\[
\xymatrix{
\triangle \times \Sigma_\triangle \ar[r]^{\Psi_{\triangle}} \ar[rd]_{\emph{projection}} & X_{\triangle} \ar[d]^{f} \\
& \triangle \subset Y
}
\]
  \item the restriction
  \[
  \Psi_\triangle\vert_{\mathring\triangle \times \Sigma_\triangle} : \mathring\triangle \times \Sigma_\triangle \to X_{\mathring\triangle}
  \]
  is a homeomorphism making $f$ a fiber bundle over $\mathring\triangle$.
\end{itemize}
\end{lemma}

\begin{proof}
For $\Sigma_\triangle$, take the fiber over the center of $\triangle$, and the rest can be verified easily.
\end{proof}

\subsection{Connected maps.}

\begin{definition}
\label{def:connectedmap}
Let $f: X \to Y$ be a continuous map of topological spaces. It is called \emph{connected} if the fibers $f^{-1}(z)$, $z \in Z$, are (nonempty and) path-connected. Every map $f$, connected or not, can be factored as
\[
X \overset{\widetilde{f}}\to \widetilde{Y} \to Y,
\]
with $\widetilde{f}$ connected, and with $\widetilde{Y}$ being the space of path-connected components of the fibers of $f$ (topologized by the finest topology making $\widetilde{f}$ continuous). The map $\widetilde{f}$ is called the \emph{associated connected map}.
\end{definition}

If $f$ is a PL map of polyhedra, then $\widetilde{f}$ is also PL, and $\widetilde{Y}$ is a polyhedron having the same dimension as $f(X)$.

\begin{lemma}
\label{lem:homology-les}
Let $f: X \to Y$ be a connected PL map of polyhedra. If $Y$ is connected then $X$ is connected.
\end{lemma}

\begin{proof}
Let $\gamma : [0,1] \to Y$ be a path in the base. Fix a simplicial structure of $Y$ for which $f$ is simplicial.
Let us build a path $\widetilde{\gamma}: [0,1] \to X$ covering $\gamma$ in the following weak sense: there is a monotone reparametrization map $r: [0,1] \to [0,1]$ such that $f(\widetilde{\gamma}(t)) = \gamma(r(t))$. First, split $\gamma$ into arcs each of which belongs to a single cell of $Y$. Without loss of generality, there are finitely many of these arcs (this can be achieved by homotoping $\gamma$ slightly, while fixing endpoints). For each such arc $[t',t''] \to Y$, one can lift $\gamma$ by Lemma~\ref{lem:plbundle}. 
If $\gamma$ is lifted independently over $[t',t]$ and $[t,t'']$, the two lifted patches can be connected inside the fiber $f^{-1}(\gamma(t))$. This is how $\widetilde{\gamma}$ can be built. For the assertion of the lemma, having two points $x, x' \in X$, one can connect $f(x)$ to $f(x')$ in the base, and lift the path as above. The endpoints of the lifted path can be connected to $x$ and $x'$ in the corresponding fibers. This proves that $X$ is connected.
\end{proof}

\subsection{Foliations.}

\begin{definition}
\label{def:foli}
Let $\Sigma$ be a topological space. We use the word \emph{foliation} to denote a continuous map $p: \Sigma \to Z$ to a graph (finite $1$-dimensional simplicial complex), in the sense that $\Sigma$ is foliated by the fibers $p^{-1}(z)$, $z \in Z$ (the \emph{leaves}).
\end{definition}

This is a non-standard use of the word ``foliation''. We could have used the word ``slicing'' as well in this context.

\begin{definition}
\label{def:folisimple}
Let $\Sigma$ be a polyhedron. We say a foliation $p : \Sigma \to Z$ is \emph{simple} if it is a connected PL map.
\end{definition}


For a foliation $p$ of a compact metric space $\Sigma$, recall the notation $\W(p) = \sup\limits_{z \in Z} \diam p^{-1}(z)$ for its width.
The next lemma shows that, in a sense, any its foliation of width $<1$ can be ``simplified'' while keeping its width $<1$.

\begin{lemma}
\label{lem:simplify}
If a metric polyhedron $\Sigma$ admits a foliation of width $<1$, then it also admits a simple foliation width $<1$.
\end{lemma}

\begin{proof}
Let $p: \Sigma \to Z$ be a foliation of width $<1$. Subdivide $Z$ finely so that the preimage of the open star\footnote{The \emph{open star} of a vertex of a simplicial complex is the union of the relative interiors of all faces containing the given vertex. In a graph, the open star of a vertex is the vertex itself together with all incident open edges.} $S_v$ of every vertex $v \in Z$ has diameter $<1$. Use the simplicial approximation theorem to approximate $p$ by a simplicial (for some subdivision of $\Sigma$) map $p'$ such that for each $x \in \Sigma$, $p'(x)$ belongs to the minimal closed cell of $Z$ containing $p(x)$. It follows that each fiber of $p'$ is contained in $p^{-1}(S_v)$ for some $v \in Z$, so $p'$ has width $<1$.

Next, replacing $p'$ by the associated connected map $\widetilde{p}'$ (which is also PL), we arrive at the situation where the leaves $(\widetilde{p}')^{-1}(z)$ are (nonempty and) connected for all $z \in Z$, and have diameter $<1$.
\end{proof}

\subsection{Interpolation lemma.}

\begin{definition}
\label{def:inter}
Let $\Sigma$ be a topological space, and let $p_0: \Sigma \to Z_0$, $p_1: \Sigma \to Z_1$ be its foliations. An \emph{interpolation} between these is a family of foliations $p_t: \Sigma \to Z_t$, $t \in [0,1]$, continuous in the following sense.
\begin{itemize}
  \item There is a $2$-dimensional polyhedron $Z_{[0,1]}$ together with a \emph{parametrization} map $\pi: Z_{[0,1]} \to [0,1]$, such that $\pi^{-1}(t) = Z_t \subset Z_{[0,1]}$.

  \item There is a continuous map $P : [0,1] \times \Sigma \to Z_{[0,1]}$ fibered over $[0,1]$, and giving $p_t$ when restricted over $\{t\}$:

\begin{minipage}{0.35\textwidth}
\[
\xymatrix{
[0,1] \times \Sigma\ar[r]^{P} \ar[rd]_{\text{projection}} & Z_{[0,1]} \ar[d]^{\pi} \\
& [0,1]
}
\]
\end{minipage}
\begin{minipage}{0.35\textwidth}
\[
\xymatrix{
\{t\} \times \Sigma\ar[r]^{p_t} \ar[rd]_{\text{projection}} & Z_t \subset Z_{[0,1]} \ar[d]^{\pi} \\
& \{t\}
}
\]
\end{minipage}
\end{itemize}
\end{definition}

\begin{lemma}
\label{lem:inter}
Let $\Sigma$ be a metric polyhedron of topological complexity $\beta = \tc(\Sigma)$, and let $p_0 : \Sigma \to Z_0$, $p_1 : \Sigma \to Z_1$ be simple foliations. It is possible to interpolate between them through simple foliations of width at most $(\beta+2)\W(p_0) + (\beta+1)\W(p_1)$.
\end{lemma}

We only outline the proof, since a more general statement will be proved in the next subsection. However, this outline illustrates the main method of this section.





\begin{lemma}
\label{lem:filter}
Given a (finite) connected graph $Z$ (viewed as a topological space), there is a filtration by closed subspaces $Z^{(t)} \subset Z$, $t \in [0,1]$, such that
\begin{itemize}
  \item $Z^{(t)} = \alpha^{-1}([0,t])$, for some continuous function $\alpha : Z \to [1/2,1]$;
  \item $Z^{(1/2)} = \alpha^{-1}(1/2)$ consists of a single point;
  \item every preimage $\alpha^{-1}(t)$, $t \in [1/2,1]$, consists of finitely many points (informally, this condition says that $Z^{(t)}$ depends continuously on $t$).
%
\end{itemize}
One can also consider a satellite filtration by open subspaces $\mathring Z^{(t)} = \bigcup\limits_{t' \in [0,t)} Z^{(t')} = \alpha^{-1}([0,t))$.
\end{lemma}

\begin{proof}
Such a filtration can be constructed using
\[
\alpha(z) = \frac{\dist_Z(z_0, z)}{2\sup\limits_{z' \in Z} \dist_Z(z_0, z')} + 1/2
\]
for any fixed point $z_0 \in Z$ and any metrization of $Z$. 
\end{proof}

\begin{proof}[Outline of the proof of Lemma~\ref{lem:inter}]
We can assume $\Sigma$ connected (by dealing with each connected component separately).

The graph $Z_1$ is connected, since $\Sigma$ is connected, and $p_1$ is simple (hence surjective).
Filter $Z_1$ as in Lemma~\ref{lem:filter}: $Z_1^{(0)} \subset \ldots \subset Z_1^{(t)} \subset \ldots \subset Z_1^{(1)}$, $t \in [0,1]$.
%
%
We interpolate between $p_0$ and $p_1$ through foliations $p_t : \Sigma \to Z_t$, which can be roughly described as follows. To get a picture of $p_t$, first you draw the fibers of $p_1$ over $Z_1^{(t)}$. Then in the remaining room we draw the fibers of $p_0$ (their parts that fit). The resulting picture is interpreted as a foliation by connected leaves, and we call it $p_t$ (see Figure~\ref{fig:interfoli}).

\begin{figure}[ht]
  \centering
  \includegraphics[width=0.85\textwidth]{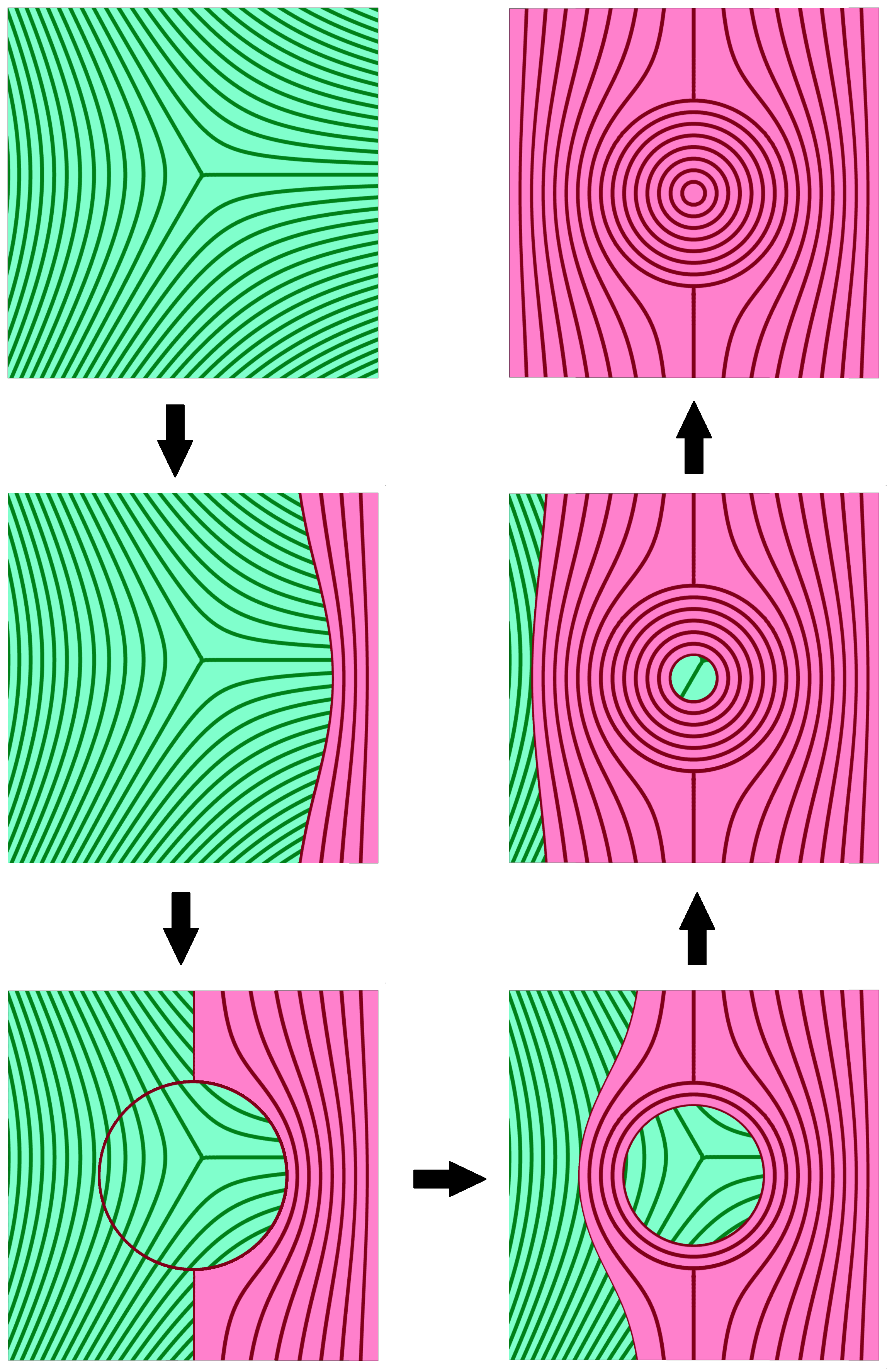}
  \caption{Interpolation between foliations. Each rectangle represents a foliation of $\Sigma$, given by a map to a graph. The foliations $p_0$ and $p_1$ are pictured in green and red, respectively}
  \label{fig:interfoli}
\end{figure}

Let us rigorously describe the space of leaves $Z_t$ and the foliation map $p_t$.

\begin{itemize}
  \item Define $Z_{0}^{(t)}$, $t \in [0,1]$, as the minimal closed subspace of $Z_0$ such that $p_0^{-1}(Z_{0}^{(t)}) \cup p_1^{-1}(\mathring Z_{1}^{(t)}) = \Sigma$; in other words,
  \[
  Z_{0}^{(t)} = p_0\left(\Sigma \setminus p_1^{-1}(\mathring Z_1^{(t)})\right).
  \]
    We write $\Sigma^{(t)} = \Sigma \setminus p_1^{-1}(\mathring Z_1^{(t)})$ for short.
  \item The map $p_0\vert_{\Sigma^{(t)}} : \Sigma^{(t)} \to Z_{0}^{(t)}$ might not have all fibers connected, so we factor it through its associated connected map:
\[
\Sigma^{(t)} \overset{\widetilde p_{0}^{(t)}}\to \widetilde Z_{0}^{(t)} \to Z_{0}^{(t)}.
\]
  \item The graph $Z_t$ is defined as
\[
\left(\widetilde Z_{0}^{(t)} \sqcup Z_{1}^{(t)} \right)/ {\overset{t}\sim},
\]
where $\overset{t}\sim$ is the following equivalence relation. Let us write $z \overset{t}\approx z'$ if $z \in \widetilde Z_{0}^{(t)}$, $z' \in Z_{1}^{(t)}$, and $(\widetilde p_{0}^{(t)})^{-1}(z)$ intersects $p_{1}^{-1}(z')$.
Define $\overset{t}\sim$ to be the transitive closure of $\overset{t}\approx$.
    There are natural maps $\iota_{0}^{(t)} : \widetilde Z_{0}^{(t)} \to Z_t$ and $\iota_{1}^{(t)} : Z_{1}^{(t)} \to Z_t$.
  \item The map $p_t : \Sigma \to Z_t$ is defined as
\[
  p_t(x) = \begin{cases}
             \iota_{1}^{(t)}(p_1(x)), & \mbox{if } p_1(x) \in Z_{1}^{(t)} \\
             \iota_{0}^{(t)}\left(\widetilde p_{0}^{(t)}(x)\right), & \mbox{otherwise}.
           \end{cases}
\]
    Observe that for $t=0,1$ this agrees with the original foliations $p_0$ and $p_1$.
\end{itemize}

This describes the intermediate foliations $p_t$, but in order to describe the interpolation completely we also need to explain how the graphs $Z_t$ assemble into a $2$-complex $Z_{[0,1]}$, and how the maps $p_t$ assemble into a continuous map $P : [0,1] \times \Sigma \to Z_{[0,1]}$. We do not give these details here, because a more general construction will be explained in the next subsection.

To finish the proof, we need to bound the size of the fibers of $p_t$. Why could it be possibly large? Because in the process of interpolating some vertices of the target graph merged under the $\overset{t}\sim$-identification, so multiple fibers of $p_0$ and $p_1$ might have been united. Consider a fiber of $p_t$. For this fiber, consider the longest chain of identifications
\[
z_0 \overset{t}\approx z_1' \overset{t}\approx z_1 \overset{t}\approx z_2' \overset{t}\approx \ldots
\]
with $z_j \in \widetilde Z_0^{(t)}$, and with $z_j' \in Z_1^{(t)}$ all distinct. Suppose it has more than $1+\tc(\Sigma)$ elements of $Z_1^{(t)}$. To every subchain $z_j' \overset{t}\approx z_j \overset{t}\approx z_{j+1}'$ assign an arc inside $(\widetilde p_{0}^{(t)})^{-1}(z_j)$ connecting some two points $x \in p_1^{-1}(z_j')$ and $y \in p_1^{-1}(z_{j+1}')$ (see Figure~\ref{fig:chain}). This arc represents an element of relative homology $H_1(\Sigma, \Sigma_1)$, where we denoted $\Sigma_1 = p_1^{-1}(Z_1^{(t)})$.

\begin{figure}[ht]
  \centering
  \includegraphics[width=0.8\textwidth]{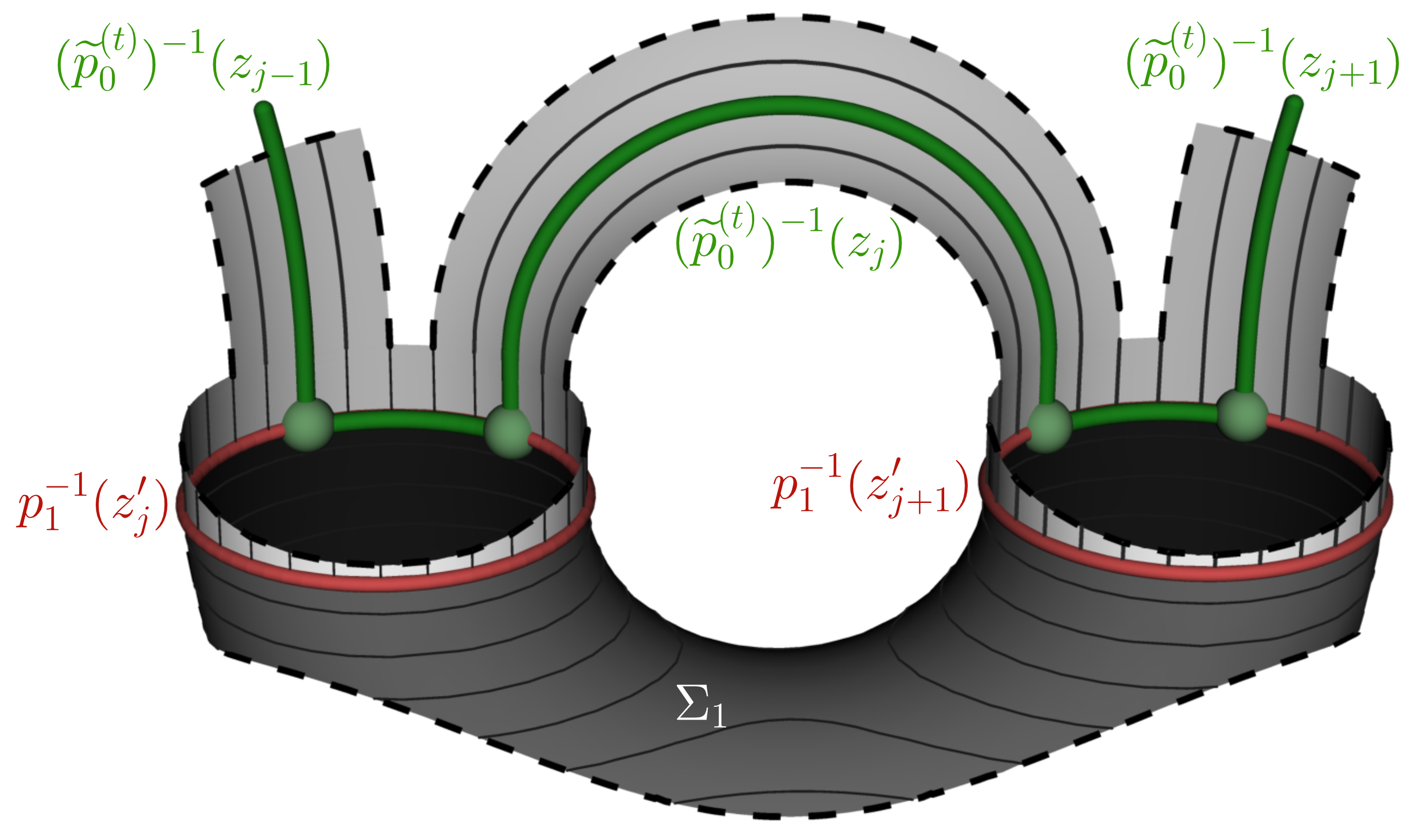}
  \caption{Chain of $\overset{t}\approx$-identifications}
  \label{fig:chain}
\end{figure}

Recall that $\Sigma_1$ is connected by Lemma~\ref{lem:homology-les}, so $\rk H_1(\Sigma, \Sigma_1) \le \tc(\Sigma)$. There must be a relation between the classes of those arcs in $H_1(\Sigma, \Sigma_1)$. It follows that some $z_j'$ repeats in the chain, which proves that such a chain has at most $1+\tc(\Sigma)$ elements of $Z_1^{(t)}$, hence at most $2+\tc(\Sigma)$ elements of $\widetilde Z_0^{(t)}$. We conclude that the diameter of a fiber of $p_t$ is at most $(\beta+2)\W(p_0) + (\beta+1)\W(p_1)$. This finishes the proof if we measure the topological complexity with $\tc(\cdot)$. For the modified complexity $\tc'(\cdot)$, one can assign a class in $H^1(\Sigma)$ to each element $z_j'$ (represented by the cochain that counts intersections with $p_1^{-1}(z_j')$). One needs to verify that there is just one linear dependence between them (coming from the $0$-cochain equal to the characteristic function of $\Sigma_1$), and that their products vanish; this will imply that some $z_j'$ must repeat.
\end{proof}

\subsection{Parametric interpolation lemma.}

\begin{definition}
\label{def:parafoli}
Let $\Sigma$ be a topological space, and let $\pi: Z_K \to K$ be a map of polyhedra such that every fiber is a (nonempty and) connected graph. A continuous map $P : K \times \Sigma \to Z_K$ is called a \emph{parametric foliation over $K$}, or a \emph{family of foliations parametrized by $K$}, if the composition $\pi \circ P : K \times \Sigma \to K$ is the projection onto the first factor:
\[
\xymatrix{
K \times \Sigma\ar[r]^{P} \ar[rd]_{\text{projection}} & Z_K \ar[d]^{\pi} \\
& K
}
\]
We call $Z_K$ the \emph{space of leaves}, and $\pi$ the \emph{parametrization map}. For $s \in K$, the restriction $P\vert_{\{s\} \times \Sigma}$ can be viewed as a foliation $p_s : \Sigma \to \pi^{-1}(s)$, and we think of $P$ as the family of foliations $p_s$ parametrized by $s \in K$. We say that $P$ is \emph{simple} if it is PL and connected.
\end{definition}

For a parametric foliation $P : K \times \Sigma \to Z_K$ of a metric space $\Sigma$, we keep using the notation $\W(P) = \sup\limits_{z \in Z_K} \diam P^{-1}(z)$ for the width.

\begin{definition}
\label{def:parainter}
Let $\Sigma$ be a topological space.
  Let $P_0 : K \times \Sigma \to Z_K$ be a family of foliations, and let $p_1 : \Sigma \to Z_1$ be another foliation. An \emph{interpolation} between them is a parametric foliation $P : CK \times \Sigma \to Z_{CK}$ over the cone $CK = ([0,1] \times K)/(\{1\} \times K)$, restricting to $P_0$ over the base $\{0\} \times K$ of $CK$, and to $p_1$ over the apex of $CK$.
\end{definition}

%

We are in position to prove the principal lemma of this section.

\begin{lemma}[Parametric interpolation]
\label{lem:parainter}
Let $\Sigma$ be a metric polyhedron of topological complexity $\beta = \tc(\Sigma)$. Let $P_K : K \times \Sigma \to Z_K$ be a family of simple foliations over a $d$-dimensional complex $K$, and let $p_1 : \Sigma \to Z_1$ be a simple foliation. It is possible to interpolate between $P_K$ and $p_1$ via a simple family $CK \times \Sigma \to Z_{CK}$ of width at most $(\beta+2)\W(P_0) + (\beta+1)\W(p_1)$.
\end{lemma}
%

\begin{proof}
We can assume $\Sigma$ connected (by dealing with each connected component separately).

The parametric foliation $P_K$ splits into simple foliations $p_s : \Sigma \to Z_s$, where $Z_s = \pi^{-1}(s)$, $s \in K$, $\pi : Z_K \to K$ is the parametrization of the foliation base.

The proof idea is simple: for each $s \in K$, interpolate between $p_s$ and $p_1$ as in Lemma~\ref{lem:inter}, and make sure that the interpolation depends nicely on $s$, in order to assemble them altogether to a parametric interpolation. The details are pretty technical, and now we write them out.

The graph $Z_1$ is finite and connected, since $\Sigma$ is compact and connected, and $p_1$ is simple (hence surjective).
Filter $Z_1$ as in Lemma~\ref{lem:filter}: $Z_1^{(0)} \subset \ldots \subset Z_1^{(t)} \subset \ldots \subset Z_1^{(1)}$, $t \in [0,1]$.
%
%
We interpolate between $P_K$ and $p_1$ via a family $P : CK \times\Sigma \to Z_{CK}$ to be described. With a little abuse of notation, we use coordinates $(t,s) \in [0,1] \times K$ on $CK$, with a convention that all points $(1,s)$ are identified with the apex of $CK$. 
The restriction $P\vert_{\{(t,s)\} \times \Sigma}$ is a foliation $p_{(t,s)} : \Sigma \to Z_{(t,s)}$, which can be pictured as follows. First, draw the fibers of $p_1$ over $Z_1^{(t)}$; then fill in the remaining room with the fibers of $p_s$ (with their parts that fit). The resulting picture is interpreted as a foliation by connected leaves, and we call it $p_{(t,s)}$.

We now describe $P : CK \times\Sigma \to Z_{CK}$ formally.

\begin{itemize}
  \item Define
  \begin{alignat*}{2}
    P_0 :&\quad& [0,1) \times K \times \Sigma &\to [0,1) \times Z_K \\
    && (t,s,x) &\mapsto (t, p_s(x)) \\
    P_1 :&& CK \times \Sigma &\to CK \times Z_1 \\
    && (c,x) &\mapsto (c, p_1(x))
  \end{alignat*}

  \item Define
  \[
  \Z_1 = \bigcup_{(t,s) \in CK} Z_1^{(t)}  \subset CK \times Z_1
  \]
  where we think of $Z_1^{(t)}$ as sitting in $\{(t,s)\} \times Z_1$. The interior of $\Z_1$ is
  \[
  \mathring\Z_1 = \bigcup_{(t,s) \in CK} \mathring Z_1^{(t)}  \subset CK \times Z_1.
  \]
  Define
%
  \[
  \Sig = \left([0,1) \times K \times \Sigma\right) \setminus P_1^{-1}(\mathring \Z_1)  \subset [0,1) \times K \times \Sigma
  \]
and
  \[
  \Z_0 = P_0\left( \Sig \right) \subset [0,1) \times Z_K.
  \]


  \item The map $P_0\vert_{\Sig}$ might not be connected, so we factor it through its associated connected map:
\[
\Sig \overset{\widetilde P_{0}}\to \widetilde{\Z_{0}} \to \Z_0.
\]

  \item The space of leaves is
  \[
   Z_{CK} = \left(\widetilde{\Z_{0}} \sqcup \Z_{1} \right)/ {\sim},
  \]
  where $\sim$ is the following equivalence relation. Let us write $z \approx z'$ if $z \in \widetilde{\Z_{0}}$, $z' \in \Z_{1}$, 
  and $\widetilde P_{0}^{-1}(z)$ intersects $P_{1}^{-1}(z')$, as subsets of $CK \times \Sigma$. (Recall our convention for coordinates in a cone, in which $[0,1) \times K \subset CK$.)
  Define $\sim$ to be the transitive closure of $\approx$.
  There are natural maps $\iota_{0} : \widetilde{\Z_{0}} \to Z_{CK}$ and $\iota_{1} : \Z_{1} \to Z_{CK}$.
  \item The parametric foliation $P$ is defined as
  \begin{align*}
    P : CK \times \Sigma &\to Z_{CK} \\
    \xi &\mapsto \begin{cases}
             \iota_{1}(P_1(\xi)), & \mbox{if } P_1(\xi) \in \Z_{1} \\
             \iota_{0}\left(\widetilde P_{0}(\xi)\right), & \mbox{otherwise}.
           \end{cases}
  \end{align*}
    It is easy to see that $P$ indeed interpolates between $P_K$ and $p_1$.
\end{itemize}

Clearly, $P$ is connected. It is rather technical but straightforward to make sure that $P$ is PL.

The analysis of the width was already done in Lemma~\ref{lem:inter}. Any foliation from the family $P$ interpolates between certain $p_s$, $s \in K$, and $p_1$, as in the construction of Lemma~\ref{lem:inter}. Therefore, $\W(P) \le (\beta+2)\W(P_0) + (\beta+1)\W(p_1)$.
\end{proof}


\subsection{Waist of a PL map}
Finally, we are ready to prove the main theorem of this section.

\begin{theorem}
\label{thm:waist}
Let $f: X \to Y^m$ be a PL map from a metric polyhedron $X$ to an $m$-dimensional polyhedron $Y$. Let $\beta = \tc(f)$ be its topological complexity, that is, $\beta = \sup\limits_{y \in Y} \tc(f^{-1}(y))$.
Then there is a fiber $X_y = f^{-1}(y)$ of Urysohn width
\[
\UW_{1}(X_y) \ge c(m,\beta) \UW_{m+1}(X),
\]
for some positive constant $c$ depending only on $m$ and $\beta$.
\end{theorem}

\begin{proof}
Replacing $f$ with its associated connected map, we can assume that $f$ is connected. Even if $f$ is not a fiber bundle, still locally this is almost the case by Lemma~\ref{lem:plbundle}. For each simplex $\triangle \subset Y$ (of any dimension) in a fine triangulation of $Y$, the map $f$ can be ``almost'' trivialized over $\triangle$ via a PL map
\[
\Psi_\triangle: \triangle \times \Sigma_\triangle \to X_{\triangle},
\]
for some polyhedron $\Sigma_\triangle$; this map is a genuine trivialization over $\mathring\triangle$, the relative interior of $\triangle$. For each $y \in \mathring\triangle$, this map induces the distance function $d_y^\triangle$ on $\Sigma_\triangle$ defined as
\[
d_y^\triangle(x,x') = \dist_X(\Psi_\triangle(y,x), \Psi_\triangle(y,x')).
\]
Refining the triangulation of $Y$ if needed, we can assume that all metrics $d_y^\triangle$ over $y \in \mathring\triangle$ are $\eps$-close to one another in the following sense: the ``layers'' $\Psi_\triangle(\triangle \times \{x\})$ have diameter less than $\eps/2$ for all $x \in \Sigma_\triangle$, hence for any $x,x' \in \Sigma_\triangle$ and any $y,y' \in \mathring\triangle$ we have $|d_y^\triangle(x,x') - d_{y'}^\triangle(x,x')| \le \eps$.

Suppose that $\UW_{1}(X_{y}) < w_0$, for all $y \in Y$, with $w_0 = c(m, \beta) \UW_{d+1}(X)$ to be specified later. We get a foliation of $X_{y}$ of width less than $w_0$, which can be assumed simple by Lemma~\ref{lem:simplify}. The idea of the proof is to pick a dense discrete set of points in $Y$, and use those foliations to build a map $F : X \to Z^{m+1}$ of controlled width.
This is done inductively on skeleta of $Y$.

At the zeroth step, for each vertex $v$ of $Y$, pick a simple foliation $F_v : X_v \to Z_v$ of width less than $w_0$.

%

At the $k\textsuperscript{th}$ step, $1 \le k \le m$, we assume that we already defined $F_{k-1}: X_{Y^{(k-1)}} \to Z_{Y^{(k-1)}}$, over the $(k-1)$-skeleton of $Y$, of width less than $w_{k-1}$, and we need to extend it over $Y^{(k)}$. Take a $k$-simplex $\triangle \subset Y$, and consider the corresponding ``trivialization'' $\Psi_\triangle : \triangle \times \Sigma_\triangle \to X_{\triangle}$. Pick a point $y$ in the relative interior of $\triangle$, and a simple foliation $p_y$ of $\Sigma_\triangle$ of $d_y^\triangle$-width $<w_0$. We would like to use Lemma~\ref{lem:parainter} to build a parametric foliation $P_\triangle : \triangle \times \Sigma_\triangle \to Z_\triangle$ interpolating between $p_y$ and the family of foliations
\[
\partial \triangle \times \Sigma_\triangle \overset{\Psi_\triangle}\to X_{\partial \triangle} \overset{F_{k-1}} \to Z_{Y^{(k-1)}}
\]
(here $\partial$ denotes the relative boundary). In order to apply that lemma, we need to fix a metric on $\Sigma_\triangle$, so we use $d_y^\triangle$ (recall that they are all $\eps$-close). We get a map $P_\triangle : \triangle \times \Sigma_\triangle \to Z_\triangle$ of width less than $(\beta+2)w_{k-1} + (\beta + 1)w_0$. The desired map $F_\triangle : X_\triangle \to Z_\triangle$ that we are looking for is already defined over $\partial \triangle$, so we specify it over $\mathring\triangle$:
\[
X_{\mathring\triangle} \overset{\Psi_\triangle^{-1}}\to \mathring\triangle \times \Sigma \overset{P_\triangle}\to Z_\triangle.
\]
The resulting map $F_\triangle$ is continuous. Repeating this over all $k$-simplices we get the map $F_k : X_{Y^{(k)}} \to Z_{Y^{(k)}}$ of width less than
\[
w_k = (\beta+2)w_{k-1} + (\beta + 1)w_0 + \eps.
\]
As $\eps \to 0$, the solution of this recurrence tends to
\[
w_k = (2(\beta+2)^k - 1)w_0.
\]
Therefore, $\UW_{m+1}(X) \le (2(\beta+2)^m - 1) c(m, \beta) \UW_{m+1}(X)$. Hence, for each $c < \frac{1}{2(\beta+2)^m - 1}$, there is a fiber $X_{y(c)}$ of width at least $c \UW_{m+1}(X)$. Finally, send $c \to \frac{1}{2(\beta+2)^m - 1}$, pick a limit point $\bar y$ of $\{y(c)\}$, and note that $\UW_{1}(X_{\bar y}) \ge \frac{\UW_{m+1}(X)}{2(\beta+2)^m - 1}$ by upper semi-continuity of width (Lemma~\ref{lem:semicontinuous}).
\end{proof}

This proof gives the value $c = \frac{1}{2(\beta+2)^m - 1}$. A more careful analysis of the proof leads to a much better value, namely $c = \frac{1}{2\beta m + m^2 + m + 1}$, which we explain now.

The inductive interpolation step in the proof of Theorem~\ref{thm:waist} is done in a manner that allows us to split $Y$ into $m$-simplices (basically according to the barycentric subdivision of the triangulation used in induction), so that over each simplex $\triangle$ we have the following picture. Over the vertices of $\triangle$, we have simple foliations $p_j : \Sigma_\triangle \to Z_j$, $j = 0,1,\ldots,m$. Over a generic point of $\triangle$, we have a foliation $p : \Sigma_\triangle \to Z$ that looks as follows. First, draw the fibers of $p_m$ over $Z_m^{(t_m)}$, a subgraph of $Z_m$ (connected or empty). In the remaining room, draw (the parts of) the fibers of $p_{m-1}$ over $Z_{m-1}^{(t_{m-1})}$, a subgraph of $Z_{m-1}$. Continue in the same fashion. At the last step, fill in the remaining room with (the parts of) the fibers of $p_0$. The touching fibers of different $p_j$ get merged to a single fiber of $p$.
What is the maximal length of a chain of merged fibers? We show that it can be bounded by $2\beta m + m^2 + m + 1$.

%

\begin{lemma}
\label{lem:intersimplex}
Let $\Sigma$ be a metric polyhedron of topological complexity $\beta = \tc(\Sigma)$. Let $p_j : \Sigma \to Z_j$, $j = 0,1,\ldots,m$, be simple foliations of width at most $1$. Suppose a parametric foliation $P : \triangle \times \Sigma \to Z_\triangle$ over an $m$-simplex (restricting to $p_j$ over the $j\textsuperscript{th}$ vertex of $\triangle$) is obtained by applying Lemma~\ref{lem:parainter} inductively; that is, one first interpolates between $p_0$ and $p_1$, then between the result and $p_2$, and so on. Then the width of $P$ is at most $2\beta m + m^2 + m + 1$.
\end{lemma}

\begin{proof}
As explained above, a generic foliation $p$ in the family $P$ is obtained by drawing fibers of $p_j$ over $Z_j^{(t_j)}$, $j=0,1,\ldots, m$. We assume that every $Z_j^{(t_j)}$ is non-empty, otherwise the result follows by induction on $m$.
Denote by $\Sigma_j$ the closed subset of $\Sigma$ covered (in the foliation $p$) by the fibers of $p_j, \ldots, p_m$ (in particular, $\Sigma_0 = \Sigma$). Notice that for $1 \le j \le m$, $\Sigma_j$ consists of at most $m-j+1$ connected components, since each set $p_j^{-1} (Z_{j}^{(t_{j})})$ is connected by Lemma~\ref{lem:homology-les}. From the long exact sequence
\[
\ldots \to  H_1(\Sigma) \to  H_1(\Sigma, \Sigma_j) \to \tilde H_0(\Sigma_j) \to \ldots
\]
one finds that $\rk H_1(\Sigma, \Sigma_j) \le \rk H_1(\Sigma) + \rk \tilde H_0(\Sigma_j) \le \beta + m-j$.

We need to bound the number of fibers in a merged chain. Fix two points $x,y \in \Sigma$ in a single fiber $p^{-1}(z)$, and connect them by a path $\alpha : [0,1] \to \Sigma$ inside this fiber. For each $t$, notice which of the regions $\Sigma_j \setminus \Sigma_{j+1}$ the point $\alpha(t)$ belongs to, and write down the corresponding index $J(t)$ (here $\Sigma_{m+1}$ is assumed empty). We have a piecewise constant function $J: [0,1] \to \{0,1,\ldots,m\}$. Denote the number of its discontinuities by $D$; without loss of generality, $D$ is finite. Note that $\dist(x,y) \le D+1$. We will transform $\alpha$ (while keeping it inside the same fiber of $p$, and fixing its endpoints $x,y$) to achieve $D \le (2\beta + m + 1) m$. Consider the following property, which $\alpha$ may or may not enjoy.

\textit{Desired property.} For $1 \le j \le m$, we say that a path $\alpha$ is \emph{$j$-nice} if the superlevel set $I^{\ge j} = \{t \in [0,1] ~\vert~ J(t) \ge j\}$ consists of at most $\beta + m-j+1$ components. We say that $\alpha$ is \emph{nice} if it $j$-nice for all $1 \le j \le m$.

Suppose first $\alpha$ is not nice, and take the smallest index $j$ such that $\alpha$ is not $j$-nice. Mark a point in each component of $I^{\ge j}$, so that we have marked points $t_1, \ldots, t_k$, $k > \beta + m-j+1$. Each arc $\alpha([t_i, t_{i+1}])$ represents an element of $H_1(\Sigma, \Sigma_j)$. Recall that $\rk H_1(\Sigma, \Sigma_j) \le \beta + m-j$. It follows that some two points $\alpha(t_i)$, $\alpha(t_{i'})$ can be connected inside $p^{-1}(z) \cap \Sigma_j$. Replace $\alpha([t_i, t_{i'}])$ with this new curve. We decreased the number of components of $I^{\ge j}$. Proceeding in the same fashion, we can make $\alpha$ $j$-nice. Repeating this procedure for larger $j$ if needed, we make $\alpha$ nice.

Now that $\alpha$ is nice, we bound its number $D$ of discontinuities. Clearly, $D$ is bounded by the total number of the endpoints of all $I^{\ge j}$. Since $\alpha$ is nice,
\[
  D \le \sum_{j=1}^{m} 2 (\beta + m-j+1) = (2\beta + m + 1) m.
\]
\end{proof}
%
We remark that the improved bound still does not seem sharp. In Gromov's example (example~\ref{cl:A} of the introduction) the dependence on $\beta$ is of order $\beta^{-1/3}$ while our bound only guarantees $\beta^{-1}$.

\begin{remark}
The proof of Theorem~\ref{thm:waist} together with the estimate of Lemma~\ref{lem:intersimplex} hold with $\tc(\cdot)$ replaced by $\tc'(\cdot)$. Indeed, in the proof of Lemma~\ref{lem:intersimplex}, to each connected component of $p^{-1}(z) \cap \Sigma_j$ that $\alpha$ meets, one can assign a class in $H^1(\Sigma)$ in a way so that their products vanish, and their linear dependencies form at most $(m-j+1)$-dimensional space (since they all arise from $0$-cochains that are characteristic functions of the connected components of $\Sigma_j$).
\end{remark}

\section{Fibered manifolds with topologically trivial fibers of small width}
\label{sec:sasha}

The following result generalizes example~\ref{cl:D} from the introduction.

\begin{theorem}
\label{thm:bundle}
For any non-negative integers $m$, $k$, and any $\eps > 0$, there exists a map $X \to Y$ such that
\begin{itemize}
  \item $X = F \times Y$, and the map is the trivial fiber bundle $F \times Y \to Y$;
  \item $Y$ and $F$ are closed topological balls of dimensions $m$ and $mk+m+k$, respectively;
  \item $X$ is endowed with a riemannian metric with $\UW_{n-1}(X) \ge 1$, where $n = \dim X = mk+2m+k$;
  \item for each $y \in Y$, the fiber $X_y \simeq F$ has $\UW_{k+m}(X_y) < \eps$.
\end{itemize}
\end{theorem}

\begin{remark}
\label{rem:bundle}
Consider the trivial bundle $X' = F' \times Y' \to Y'$, where $Y'$ is the euclidean $m$-ball of radius $\sim \eps$, and $F'$ is the euclidean $(mk+m+k)$-ball of radius $\sim \eps$. The bundle $X$ in the theorem will be constructed in a way so that near its boundary $X$ will look exactly like $X'$. This allows to modify the construction to make $X$ a closed manifold (e.g., a sphere or a torus), or to take the connected sum with other fibrations, etc.
\end{remark}

\begin{proof}
To start with, take $Y = \mathbb{R}^m$, $F = \mathbb{R}^{mk+m+k}$, $X = F \times Y = \mathbb{R}^{mk+2m+k}$, and ignore for the moment that they are not closed balls. Let $p : X \to Y$ and $p_F : X \to F$ be the projection maps. We start from the euclidean metric on $X$, modify it, and then cut $X$ to make it compact. Then the (restricted) map $p$ will be the one we are looking for.

On the first factor $F=\mathbb{R}^{mk+m+k}$, consider the structure of the $\eps$-local join of $k$-dimensional complexes $Z_0, \ldots, Z_m$ in the sense of Definition~\ref{def:join}. The construction is based on the idea of blowing up the metric in between the $Z_i$ (cf.~\cite[Section~4]{balitskiy2020local}, where a similar idea is used). Let $\tau: F \to \triangle^m$ be the join map. We think of $\triangle^m$ as sitting in $\mathbb{R}^m$ with the center at the origin, scaled so that the inradius of $\triangle^m$ equals $3$. 
Consider the ``perturbation of the projection via the join map''
\[
p^\tau : X \to Y, \quad p^\tau = p - \tau \circ p_F.
\]

It will be useful to look at $X$ in the coordinates $\Phi = (p_F, p^\tau)$. Namely, $\Phi: X \to X$ is the map given by $\Phi(x) = (p_F(x), p^\tau(x)) \in F \times Y = X$; its inverse is given by $x\mapsto (p_F(x), p(x) + \tau\circ p_F(x))$. It follows that the fibers of $p^\tau$ are PL homeomorphic to $F$.


Let $\phi_1 : [0,+\infty) \to \mathbb{R}$ be a monotone cut-off function that equals 1 on $[0,1]$ and $0$ on $[1.1,\infty)$. Denote by $\phi^k_r:\mathbb{R}^{k}\to \mathbb{R}$ an $r$-sized bump function $\phi^k_r(x):=\phi_1(|x|/r)$; here $|\cdot|$ is the euclidean norm in $\mathbb{R}^{k}$. Let $g_X^{\text{euc}}, g_Y^{\text{euc}}, g_F^{\text{euc}}$ be the standard metrics on the corresponding euclidean spaces, viewed as symmetric $2$-forms. To define a new metric on $X$ we take $g_X^{\text{euc}}$, blow it up transversely to $(p^\tau)^{-1}(x)$ for $x$ close to the origin of $\mathbb{R}^m$, and squeeze everywhere else. Formally,
\[
g_X = \Phi^* g_X', \text{  where } g_X' = \eps g_X^{\text{euc}} + (1-\eps) (\phi^{mk+m+k}_2 g_F^{\text{euc}}) \times (\phi^m_2 g_Y^{\text{euc}}).
\]
In order for this to be well-defined, one might want to approximate $\Phi$ by a smooth map. From now on, we assume that $X$ is endowed with metric $g_X$.
To make $X$ compact, one can replace it by its subset $B_{3}^{g_F^\text{euc}}(0) \times B_{3+m}^{g_Y^\text{euc}}(0)$. Radius $3+m$ here is chosen so that the $2.2$-neighborhood of $\triangle^m$ is covered by $p(X)$. We write $X'$ for the space $\Phi(X)$ with metric $g_X'$; clearly, $X$ and $X'$ are isometric.


Figure~\ref{fig:sasha} depicts the case $m=1, k=0$: there, $X = \mathbb{R}^2$ is sliced by lines $p^{-1}(y)$ (bold black curves in the figure), each of which is the local join of a green point set $Z_0$ and a blue point set $Z_1$. On the left, the geometry of $g_X$ is depicted by stretching $X$ along the vertical direction, so that it corresponds to the value of $p^\tau$. On the right, one sees $X$ in the coordinates $\Phi = (p_f, p^\tau)$, with the pinching in the region where $|p^\tau(x)|>2$.



\begin{figure}[ht]
  \centering
  \includegraphics[width=0.7\textwidth]{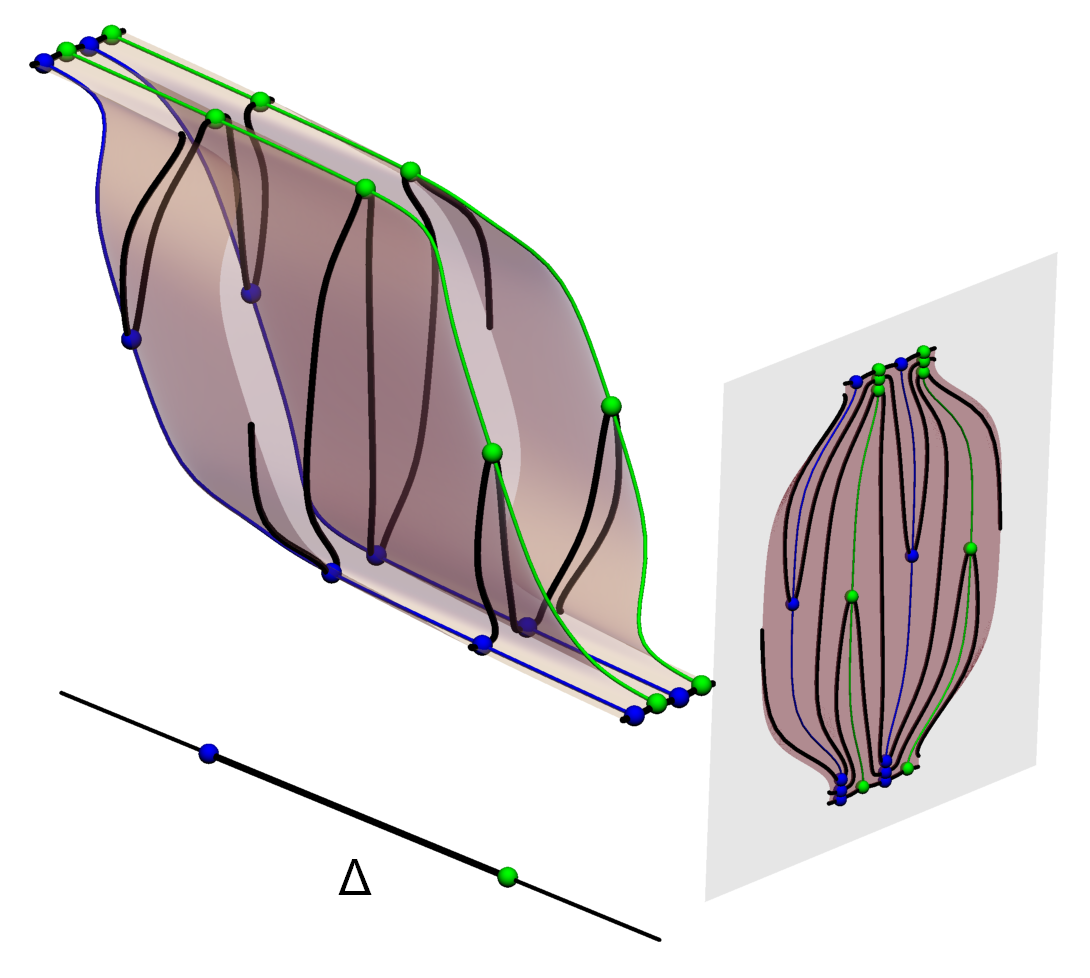}
  \caption{On the left: the map $p: X \to Y$, with $X$ stretched vertically according to the values of $p^\tau$. On the right: $X$ viewed in the coordinates $(p_f, p^\tau)$}
  \label{fig:sasha}
\end{figure}

Now let us verify the claimed properties of the metric $g_X$. To see that $\UW_{n-1}(X) \ge 1$, note that the unit ball $B_{1}^{g_X'}(0)$ is just the usual euclidean ball, and its width is $>1$.

Finally, we show that the fibers of $p$ have small width. 
Consider a fiber $X_y = p^{-1}(y)$, $y \in Y$, and the restriction of $g_X$ on it. It equals $\eps g_F^{\text{euc}}$ plus a term supported in $\tau^{-1}(B_{2.2}^{g_Y^{\text{euc}}}(y))$. The ball $B_{2.2}^{g_Y^{\text{euc}}}(y)$ does not reach one of the faces $v_i^\vee$ of $\triangle^m$. We would like to use the retraction $\pi_i$ (as in the discussion after Definition~\ref{def:join}) to map $p^{-1}(y)$ to $Z_i$; this is not possible for the points in the dual complex $Z_i^\vee$, which is entirely contained in the squeezed zone, so we will not lose much if we just send it to a single point. Here is the map witnessing $\UW_{k+m}(X_y) \lesssim \eps$:
\begin{align*}
X_y \simeq F &\to (Z_i \times \triangle^m)/(Z_i \times v_i^\vee) \\
x &\mapsto
\begin{cases}
  (\pi_i(x), \tau(x)), & \mbox{if } x \notin Z_i^\vee \\
  \star, & \mbox{otherwise}.
\end{cases}
\end{align*}
where $\star$ denotes the pinched copy of $Z_i \times v_i^\vee$ in the quotient. The fiber of this map over~$\star$ is $\eps$-small since the metric is squeezed around $Z_i^\vee$. Consider the fiber over any other point $(z,t)$ of the quotient; since it is contained in $\tau^{-1}(t)$, its $g_X$-size does not exceed its $g_F$-size; since it is contained in $\pi_i^{-1}(z)$, its $g_F$-size is $\eps$-small.
\end{proof}

\bibliography{ref}
\bibliographystyle{abbrv}
\end{document}